\documentclass[twoside,a4paper,reqno,11pt]{amsart} 
\usepackage[top=28mm,right=28mm,bottom=28mm,left=28mm]{geometry}
\usepackage{microtype}
\usepackage{amsfonts, amsmath, amssymb, amsbsy, mathrsfs, bm, latexsym, stmaryrd, hyperref, array, enumitem, textcomp, url}
\usepackage[table]{xcolor}

\renewcommand{\a}{\alpha}

\newcommand{\e}{\epsilon}
\renewcommand{\l}{\lambda}

\renewcommand{\O}{\Omega}

\newcommand{\la}{\langle}
\newcommand{\ra}{\rangle}
\newcommand{\leqs}{\leqslant}
\newcommand{\geqs}{\geqslant}

\newcommand{\vs}{\vspace{3mm}}

\makeatletter
\newcommand{\imod}[1]{\allowbreak\mkern4mu({\operator@font mod}\,\,#1)}
\makeatother

\theoremstyle{plain}
\newtheorem{theorem}{Theorem}

\newtheorem{thm}{Theorem}[section] 
\newtheorem{lem}[thm]{Lemma}
\newtheorem{prop}[thm]{Proposition}

\newtheorem*{theorem*}{Theorem} 
\newtheorem*{conj*}{Conjecture}

\theoremstyle{definition}
\newtheorem{rem}[thm]{Remark}

\begin{document}

\title[Extremely primitive affine groups]{On the classification of extremely primitive \\ affine groups}

\author{Timothy C. Burness}
\address{T.C. Burness, School of Mathematics, University of Bristol, Bristol BS8 1UG, UK}
\email{t.burness@bristol.ac.uk}

\author{Melissa Lee}
\address{M. Lee, Department of Mathematics, University of Auckland, Auckland, New Zealand}
\email{melissa.lee@auckland.ac.nz}

\date{\today} 

\begin{abstract}
Let $G$ be a finite non-regular primitive permutation group on a set $\Omega$ with point stabiliser $G_{\a}$. Then $G$ is said to be extremely primitive if $G_{\a}$ acts primitively on each of its orbits in $\Omega \setminus \{\a\}$, which is a notion dating back to work of Manning in the 1920s. By a theorem of Mann, Praeger and Seress, it is known that every extremely primitive group is either almost simple or affine, and all the almost simple examples have subsequently been determined. Similarly, Mann et al. have classified all of the affine extremely primitive groups up to a finite, but  undetermined, collection of groups. Moreover, if one assumes Wall's conjecture on the number of maximal subgroups of an almost simple group, then there is an explicit list of candidates, each of which has been eliminated in a recent paper by Burness and Thomas. So, modulo Wall's conjecture, the classification of extremely primitive groups is complete. In this paper we adopt a different approach, which allows us to complete this classification in full generality, independent of the veracity or otherwise of Wall's conjecture in the almost simple setting. Our method relies on recent work of Fawcett, Lee and others on the existence of regular orbits of almost simple groups acting on irreducible modules.
\end{abstract}

\maketitle

\section{Introduction}\label{s:intro}

Let $G \leqs {\rm Sym}(\O)$ be a finite primitive permutation group with point stabiliser $H=G_{\a} \neq 1$. Following Mann et al. \cite{MPS}, $G$ is said to be \emph{extremely primitive} if $H$ acts primitively on each of its orbits in $\O \setminus \{\a\}$. Examples include the usual action of ${\rm Sym}_{n}$ of degree $n$ and the action of ${\rm PGL}_{2}(q)$ on the projective line over $\mathbb{F}_q$. Extremely primitive groups arise naturally in several different contexts. For instance, they feature in the original constructions of some of the sporadic simple groups (in particular ${\rm J}_{2}$ and ${\rm HS}$) and they arise in the study of permutation groups with restricted movement (see \cite{PP}, for example). As far back as the 1920s, one can find important work of Manning \cite{Manning} on extremely primitive groups and they have been the main subject of several papers in recent years \cite{BPS, BPS2, BT, BTh_ep, MPS}.

In this paper, we complete the classification of the finite extremely primitive groups. Here the first main result is \cite[Theorem 1.1]{MPS} by Mann, Praeger and Seress, which reveals that every group of this form is either almost simple or affine. The sequence of papers \cite{BPS, BPS2, BTh_ep} yields a complete classification of the almost simple extremely primitive groups (see \cite[Table 1]{BTh_ep} for the complete list), so we may assume $G = V{:}H \leqs {\rm AGL}(V)$ is a primitive affine type group. Here $V = \mathbb{F}_p^d$ is a vector space with $p$ a prime and $H \leqs {\rm GL}(V)$ is irreducible. In \cite{MPS}, Mann et al. exhibit various families of extremely primitive affine groups and they are able to prove that their list is complete, up to the possibility of finitely many exceptions. Furthermore, they establish a number of important restrictions on the structure of any additional extremely primitive affine groups. For example, the problem is reduced to the case where $G$ is simply primitive (that is, primitive but not doubly transitive), $p=2$ and $H$ is almost simple.

Let $\mathcal{M}(H)$ be the set of maximal subgroups of $H$. In \cite{MPS}, Mann et al. apply asymptotic estimates on $|\mathcal{M}(H)|$ in order to obtain their affine classification up to a finite, but undetermined, list of additional groups. In this direction, a well known conjecture of G.E. Wall from the early 1960s (see \cite{Wall}) asserts that the bound $|\mathcal{M}(H)| \leqs |H|$ holds for every finite group $H$. Wall himself proved this for soluble groups, but it has recently been shown to be false in general (see \cite{Lub_wall}). However, the conclusion is still expected to be valid when $H$ is an almost simple group, but a proof remains far out of reach, even though there have been major advances in recent decades in our understanding of the subgroup structure of almost simple groups. Despite this difficulty, some strong asymptotic results have been established. For example, the main theorem of \cite{LSh96} implies that $|\mathcal{M}(H)| \leqs |H|$ for all sufficiently large alternating and symmetric groups. In \cite{MPS}, Mann et al. use the weaker bound $|\mathcal{M}(H)|<|H|^{8/5}$ for $|H| \gg 0$ (see \cite{LSh96}) to prove their classification up to finitely many exceptions. In fact, by assuming Wall's bound for  almost simple groups, they are able to determine an explicit list of additional possibilities (see \cite[Table 2]{MPS}), each of which has subsequently been eliminated by Burness and Thomas in \cite{BT}.

By combining all of the above work, we obtain a complete classification of the finite extremely primitive groups, modulo a proof of Wall's conjecture for almost simple groups. In this paper, our goal is to remove this dependence on Wall's conjecture. By adopting a different approach, we will show that the list of extremely primitive affine groups presented in \cite{MPS} is indeed complete, which constitutes the final step in this classification programme.

Our main theorem is the following. Note that in part (ii)(a), a prime divisor $r$ of $p^d-1$ is a \emph{primitive prime divisor} if $r$ does not divide $p^i-1$ for all $i=1, \ldots, d-1$ (in other words, the order of $\mbox{$p$ mod $r$}$ is $d$). 

\begin{theorem}\label{t:main}
Let $G \leqs {\rm Sym}(\O)$ be a finite primitive permutation group with point stabiliser $H$. Then $G$ is extremely primitive if and only if one of the following holds:
\begin{itemize}\addtolength{\itemsep}{0.2\baselineskip}
\item[{\rm (i)}] $G$ is almost simple and $(G,H)$ is one of the cases in \cite[Table 1]{BTh_ep}. 
\item[{\rm (ii)}] $G = V{:}H \leqs {\rm AGL}_{d}(p)$ is affine, $p$ is a prime and one of the following holds:

\vspace{1mm}

\begin{itemize}\addtolength{\itemsep}{0.2\baselineskip}
\item[{\rm (a)}] $H=Z_r.Z_e$, where $e$ divides $d$ and $r$ is a primitive prime divisor of $p^d-1$.
\item[{\rm (b)}] $p=2$ and $H = {\rm L}_{d}(2)$ with $d \geqs 3$, or $H={\rm Sp}_{d}(2)$ with $d \geqs 4$. 
\item[{\rm (c)}] $p=2$ and $(d,H) = (4,{\rm Alt}_6)$, $(4, {\rm Alt}_7)$, $(6, {\rm U}_{3}(3))$ or $(6,{\rm U}_{3}(3).2)$.
\item[{\rm (d)}] $p=2$ and $(d,H)$ is one of the following:
\[
\begin{array}{llll}
(10,{\rm M}_{12}) & (10,{\rm M}_{22}) & (10,{\rm M}_{22}.2) & (11,{\rm M}_{23}) \\
(11,{\rm M}_{24}) & (22, {\rm Co}_{3}) &  (24,{\rm Co}_{1}) & (2k, {\rm Alt}_{2k+1}) \\
(2k, {\rm Sym}_{2k+1}) &  (2\ell, {\rm Alt}_{2\ell+1}) & (2\ell, {\rm Sym}_{2\ell+1}) & (2\ell, \O_{2\ell}^{\pm}(2)) \\
(2\ell, {\rm O}_{2\ell}^{\pm}(2)) &  (8, {\rm L}_{2}(17)) & (8, {\rm Sp}_{6}(2)) & 
\end{array}
\]
where $k \geqs 2$ and $\ell \geqs 3$.
\end{itemize}
\end{itemize}
\end{theorem}

Note that the soluble examples are recorded in part (ii)(a), where $Z_n$ denotes a cyclic group of order $n$. The groups appearing in parts (ii)(b) and (ii)(c) are insoluble and doubly transitive, while each group in (ii)(d) is simply primitive and $H$ is almost simple. As noted above, Mann et al. prove that any additional extremely primitive affine group must be simply  primitive, with $p=2$ and $H$ almost simple, which provides a starting point for our work in this paper. More precisely, as we will explain in Section \ref{s:proof} below, we may assume that the socle $H_0$ of $H$ acts irreducibly on $V$ and one of the following holds:
\begin{itemize}\addtolength{\itemsep}{0.2\baselineskip}
\item[{\rm (I)}] $H_0$ is an alternating or sporadic simple group;
\item[{\rm (II)}] $H_0$ is a simple group of Lie type defined over a field of odd characteristic;
\item[{\rm (III)}] $H_0$ is a simple group of Lie type over $\mathbb{F}_2$ and $V = L(\l)$, where the highest weight $\l$ is $2$-restricted.
\end{itemize}
In addition, by combining \cite[Theorem 4.8]{MPS} with the main theorem of \cite{BT}, we can immediately eliminate any groups for which Wall's bound $|\mathcal{M}(H)| \leqs |H|$ is known to hold. In this way, we can often apply a computational approach (using {\sc Magma} \cite{Magma}) to rule out simple groups of suitably small order.

Our basic approach in Cases I, II and III relies on an elementary  observation that provides a connection to an important and widely studied problem in the representation theory of finite groups. Recall that $H$ has a regular orbit on the irreducible module $V$ if there exists a vector $v \in V$ such that the stabiliser $H_v$ is trivial. By extreme primitivity, the stabiliser $H_v$ of any nonzero vector $v \in V$ is a maximal subgroup of $H$ (since this is equivalent to the property that $H$ acts primitively on the $H$-orbit of $v$). But if $H_v=1$ then $H_v$ is non-maximal in $H$ (recall that $H$ is almost simple), whence $G$ is not extremely primitive. So our main goal is to establish the existence of a regular orbit of $H$ on the irreducible module $V$, whenever it is possible to do so. 

In the last few years, several authors have studied this existence problem in the setting we are interested in, with $H$ almost simple. For example, work of Fawcett et al. \cite{Faw1, Faw2} will quickly allow us to handle Case I above (see Proposition \ref{p:altspor}) and recent work of Lee \cite{L1} will eliminate Case II (see Proposition \ref{p:odd}). In Case III, we can apply results on maximal subgroups of exceptional groups of Lie type to reduce to the case where $H_0$ is a finite simple classical group over $\mathbb{F}_2$. Here the linear case will be handled by applying a special case of Lee's work in \cite{L2} on regular orbits of linear groups on modules in the defining characteristic. 

The remaining classical groups need closer attention and in most cases we are able to establish the existence of a regular orbit of $H$ on $V$. In order to do this, we use the trivial observation that if $H$ does not have a regular orbit on $V$ then each vector in $V$ is fixed by some element $x \in \mathcal{P}$, where $\mathcal{P}$ is the set of elements of prime order in $H$. In other words, if $C_V(x)$ denotes the $1$-eigenspace of $x$ on $V$, then $H$ has no regular orbit on $V$ if and only if 
\[
V = \bigcup_{x \in \mathcal{P}}C_V(x).
\]
Since $|V| = 2^d$, we conclude that $H$ has a regular orbit on $V$ if
\begin{equation}\label{e:2d}
|V| = 2^d > \sum_{x \in \mathcal{P}}|C_V(x)| = \sum_{i=1}^k|x_i^H| \cdot 2^{\dim C_V(x_i)},
\end{equation}
where $x_1, \ldots, x_k$ is a complete set of representatives of the $H$-classes of elements of prime order in $H$. This explains why we will be interested in deriving upper bounds on $|C_V(x_i)|$. With this aim in mind, let us observe that if $x \in H$ is nontrivial and $\a(x)$ is the minimal size of a subset $S \subseteq x^{H_0}$ with $\la H_0, x\ra = \la S \ra$, then the irreducibility of $H_0$ on $V$ forces  
\begin{equation}\label{e:cvx}
\dim C_V(x) \leqs \left(1-\frac{1}{\a(x)}\right)\dim V.
\end{equation}

In this way, we can use upper bounds on $\a(x)$ due to Guralnick and Saxl \cite{GS} to impose very strong restrictions on the highest weight $\l$, leaving a handful of low-dimensional modules to investigate in more detail. To handle these cases, we will often need to establish stronger bounds on the dimensions of the fixed spaces $C_V(x)$, either by working directly with the module or by applying bounds obtained in recent work by Guralnick and Lawther \cite{GL} at the level of the corresponding algebraic groups over the algebraic closure of $\mathbb{F}_2$.

\section{Proof of Theorem \ref{t:main}}\label{s:proof}

Let $G \leqs {\rm Sym}(\O)$ be a finite extremely primitive group with point stabiliser $H$. By \cite[Theorem 1.1]{MPS}, $G$ is either almost simple or an affine type group. The main results in \cite{BPS, BPS2, BTh_ep} provide a classification of the almost simple extremely primitive groups (see \cite[Table 1]{BTh_ep} for the complete list). Therefore, we may assume $G = V{:}H \leqs {\rm AGL}(V)$ is an affine group, where $V = \mathbb{F}_p^d$ for some prime $p$ and $H \leqs {\rm GL}(V)$ is irreducible. Moreover, in view of \cite[Theorem 1.2]{MPS}, we may assume $G$ is simply primitive, $p=2$ and $H$ is almost simple with socle $H_0$. The groups appearing in part (ii)(d) of Theorem \ref{t:main} are of this form and they are all extremely primitive by \cite[Theorem 1.3]{MPS}. So to complete the proof of Theorem \ref{t:main}, it remains to show that no additional examples arise.

Let $\mathcal{M}(H)$ be the set of maximal subgroups of $H$. The following result will be very useful.

\begin{thm}\label{t:wall}
Suppose $G$ is simply primitive, $p=2$, $H$ is almost simple and $|\mathcal{M}(H)| \leqs |H|$. Then $G$ is extremely primitive if and only if it is one of the groups recorded in part (ii)(d) of Theorem \ref{t:main}.
\end{thm}

\begin{proof}
This follows by combining \cite[Theorem 4.8]{MPS} with the main theorem of \cite{BT}.
\end{proof}

We begin by handling the cases where $H_0$ is an alternating or sporadic group. This is an easy application of the results of Fawcett et al. \cite{Faw1, Faw2} on regular orbits.

\begin{prop}\label{p:altspor}
The conclusion to Theorem \ref{t:main} holds if $H_0$ is an alternating or sporadic group.
\end{prop}

\begin{proof}
First assume $H_0 = {\rm Alt}_m$ with $m \geqs 5$. With the aid of {\sc Magma} \cite{Magma}, it is routine to check that $|\mathcal{M}(H)|\leqs |H|$ if $m \leqs 12$, so by Theorem \ref{t:wall} we may assume $m > 12$. We now apply \cite[Theorem 1.1]{Faw1}, which states that either $V = \mathbb{F}_2^d$ is the fully deleted permutation module for $H$ over $\mathbb{F}_2$ (see \cite[p.185]{KL} for the definition), or $H$ has a regular orbit on $V$. As explained in Section \ref{s:intro}, if $H$ has a regular orbit then $G$ is not extremely primitive. On the other hand, if $V$ is the fully deleted permutation module over $\mathbb{F}_2$ then $G$ is extremely primitive (see \cite[Lemma 4.2]{MPS}) and these groups are recorded in part (ii)(d) of Theorem \ref{t:main}.

Now assume $H_0$ is a sporadic group and let $\mathbb{M}$ be the Monster group. If $H_0 \ne \mathbb{M}$ then all of the maximal subgroups of $H$ are known up to conjugacy (a convenient reference is \cite{Wil}) and it is straightforward to check that $|\mathcal{M}(H)| \leqs |H|$. Therefore, we may assume $H = \mathbb{M}$. Here \cite[Theorem 1.1]{Faw2} implies that $H$ has a regular orbit on $V$ and the result follows.
\end{proof}

For the remainder, we may assume $H_0$ is a simple group of Lie type. By primitivity, $H \leqs {\rm GL}_d(2)$ is irreducible and thus $C_{{\rm GL}_d(2)}(H) \cong \mathbb{F}_{2^a}^{\times}$ for some integer $a$ dividing $d$ (see \cite[Lemma 2.10.2(i)]{KL}). This allows us to view $H$ as a subgroup of ${\rm \Gamma L}_{d/a}(2^a) = {\rm GL}_{d/a}(2^a).a$ with $H_0 \leqs {\rm GL}_{d/a}(2^a)$, and we identify $V$ as a $(d/a)$-dimensional vector space $V'$ over $\mathbb{F}_{2^a}$ on which $H$ acts semi-linearly. By \cite[Lemma 3.2]{MPS}, $H_0$ acts irreducibly on $V$. Moreover, by \cite[Lemma 3.6]{MPS} and the subsequent discussion in \cite{MPS}, we may assume that $H_0$ acts absolutely irreducibly on $V'$.

\begin{lem}\label{l:socle}
If $H_0$ has a regular orbit on $V'$, then $G$ is not extremely primitive.
\end{lem}

\begin{proof}
Suppose $0 \ne v \in V'$ and $(H_0)_v = 1$. It suffices to show that $H_v$ is a non-maximal subgroup of $H$ (as previously noted, this implies that the action of $H$ on the orbit of $v$ is not primitive). To see this, first note that $H_v \cong H_0H_v/H_0 \leqs {\rm Out}(H_0)$, so $H_v$ is a soluble subgroup of $H$ and its order divides $|{\rm Out}(H_0)|$. In \cite{LZ}, Li and Zhang determine all the maximal soluble subgroups of every almost simple group and it is routine to check that there are no examples with the required properties. Indeed, $|K|>|{\rm Out}(H_0)|$ for every soluble maximal subgroup $K$ of $H$. The result follows.
\end{proof}

\begin{prop}\label{p:odd}
The conclusion to Theorem \ref{t:main} holds if $H_0$ is a group of Lie type in odd characteristic.
\end{prop}

\begin{proof}
In view of Theorem \ref{t:wall} and Lemma \ref{l:socle}, it suffices to show that either $H_0$ has a regular orbit on $V' = \mathbb{F}_{2^a}^{d/a}$, or $|\mathcal{M}(H)| \leqs |H|$. Set $G_0 = V'{:}H_0$, so $G_0$ is a primitive affine group and $H_0$ is a simple and absolutely irreducible subgroup of ${\rm GL}(V')$. We are now in a position to apply \cite[Corollary 1.2]{L1} with respect to $G_0$, which implies that either $H_0$ has a regular orbit on $V'$, or $H_0$ belongs to a finite list of simple groups that can be read off from \cite[Table 1.2]{L1}. In each of these cases, we can compute $|\mathcal{M}(H)|$ precisely (using {\sc Magma} \cite{Magma}, for example) and then apply Theorem \ref{t:wall} to conclude. For example, if $H_0 = {\rm U}_n(r)$ with $n \geqs 3$ and $r$ odd, then this approach reduces the problem to the groups with $H_0 \in \{ {\rm U}_3(3), {\rm U}_3(5), {\rm U}_4(3) \}$ and in each case the bound $|\mathcal{M}(H)| \leqs |H|$ is easily verified.
\end{proof}

To complete the proof of Theorem \ref{t:main}, it remains to consider the case where $H_0$ is a simple group of Lie type over $\mathbb{F}_{2^e}$. As explained in \cite[Section 4]{MPS} (see Lemmas 4.3 and 4.4), we may assume that $a = e =1$, so $H_0 \leqs {\rm GL}(V) = {\rm GL}_{d}(2)$ is absolutely irreducible. In addition, we may assume $V = L(\l)$, where $\l$ is a $2$-restricted highest weight. In other words, if we fix a set $\{\l_1, \ldots, \l_{\ell}\}$ of fundamental dominant weights for $H_0$ (labelled in the usual way, see \cite{Bou}), then $\l = \sum_{i}c_i\l_i$ with $c_i \in \{0,1\}$ for all $i$. If $H_0 = {\rm U}_{\ell+1}(2)$ or $\O_{2\ell}^{-}(2)$ then there are some additional restrictions on $\l$. Namely, if $H_0$ is unitary then the fact that $V$ is defined over $\mathbb{F}_2$ implies that $\l$ is fixed under the corresponding graph automorphism of the root system of type $A_{\ell}$. In other words, if $\l = \sum_{i}c_i\l_i$ then $c_i = c_{\ell+1-i}$ for all $1 \leqs i \leqs \lfloor \ell/2\rfloor$. Similarly, if $H_0 = \O_{2\ell}^{-}(2)$ then $c_{\ell-1} = c_{\ell}$. 

\begin{prop}\label{p:ex}
The conclusion to Theorem \ref{t:main} holds if $H_0$ is an exceptional group of Lie type over $\mathbb{F}_2$.
\end{prop}

\begin{proof}
First assume $H = H_0 = E_8(2)$. The case where $V$ is the adjoint module for $H$ appears in \cite[Table 2]{MPS} and it was handled in the proof of \cite[Proposition 4.1]{BT}. Therefore, by inspecting \cite[Table A.53]{Lub} we deduce that $d \geqs 3626$. 
Recall that $H$ has a regular orbit on $V$ if the inequality in \eqref{e:2d} is satisfied, where $\mathcal{P}$ is the set of elements of prime order in $H$. Given $x \in \mathcal{P}$, let $\a(x)$ be the minimal size of a subset $S \subseteq x^H$ with $H = \la S \ra$ and recall the upper bound on $\dim C_V(x)$ in \eqref{e:cvx}. By \cite[Theorem 5.1]{GS} we have $\a(x) \leqs 11$ for all $x \in \mathcal{P}$ and thus
\[
\bigcup_{x \in \mathcal{P}}|C_V(x)|  < |H|\cdot 2^{\left(1-\frac{1}{11}\right)d} < 2^{248+\frac{10}{11}d}.
\]
But one can check that this upper bound is less than $|V| = 2^d$ for all $d \geqs 3626$, whence $H$ has a regular orbit on $V$ and we conclude that $G$ is not extremely primitive.

In the remaining cases
\[
H_0 \in \{E_7(2), E_6^{\e}(2), F_4(2), G_2(2)', {}^2F_4(2)', {}^3D_4(2)\}
\]
we have a complete description of the maximal subgroups of $H$ up to conjugacy and it is routine to check that $|\mathcal{M}(H)| \leqs |H|$. Indeed, the maximal subgroups in the latter three cases can be accessed using {\sc Magma}, while we refer the reader to \cite{BBR}, \cite{KW, Wil2} and \cite{NW}  for the groups with socle $H_0 = E_7(2)$, $E_6^{\e}(2)$ and $F_4(2)$, respectively.  Therefore, in each of these cases the result follows via Theorem \ref{t:wall}.  
\end{proof}

For the remainder, we may assume $H_0$ is a finite simple classical group over $\mathbb{F}_2$ with natural module $W$ of dimension $n$. Let $V_{{\rm adj}}$ and $V_{{\rm alt}}$ denote the largest composition factors of the adjoint module and the alternating square module $\Lambda^2(W)$ for $H_0$, respectively. 

\begin{prop}\label{p:special}
The conclusion to Theorem \ref{t:main} holds if $V =W$ or $V_{{\rm alt}}$, or if $H_0 = {\rm L}_{n}^{\e}(2)$ and $V = V_{{\rm adj}}$. 
\end{prop}

\begin{proof}
First assume $V = W$, so $\l = \l_1$, $d = n$ and $H_0 \ne {\rm U}_{n}(2)$. If $H_0 = {\rm L}_{n}(2)$ or ${\rm Sp}_{n}(2)$ then $H = H_0$ (because the highest weight $\l$ is not fixed by an involutory graph automorphism of $H_0$) and $G = V{:}H$ is $2$-transitive and extremely primitive (see \cite[Theorem 1.2(b)(i)]{MPS}). Similarly, if $H_0 = \O_{n}^{\pm}(2)$, then $H = H_0$ or $H_0.2 = {\rm O}_{n}^{\pm}(2)$ and $G$ is extremely primitive but not $2$-transitive (see \cite[Theorem 1.3(c)]{MPS}). 
Next assume $V = V_{{\rm alt}}$, so $\l = \l_2$. This case is handled in \cite[Lemma 4.5]{MPS} and the result follows. Similarly, if $H_0 = {\rm L}_{n}^{\e}(2)$ and $V = V_{{\rm adj}}$ then $\l = \l_{1}+\l_{n-1}$ and we apply \cite[Lemma 4.6]{MPS}. 
\end{proof}

\begin{rem}
In Proposition \ref{p:special}, the same conclusion holds if the given module $V$ is replaced by its dual $V^*$.
\end{rem}

\begin{prop}\label{p:linear}
The conclusion to Theorem \ref{t:main} holds if $H_0 = {\rm L}_{n}(2)$.
\end{prop}

\begin{proof}
We apply \cite[Corollary 1.4]{L2}, which implies that either $H$ has a regular orbit on $V = L(\l)$, or $(n,\l)$ is one of the cases appearing in \cite[Table 1.1]{L2} (up to duals). By inspecting the table, excluding the cases where $\l$ is not $2$-restricted or $\l \in \{\l_1, \l_2, \l_1+\l_{n-1}\}$ (in view of Proposition \ref{p:special}), we deduce that either $\l = \l_3$ and $n \in \{6,7,8,9\}$ or $\l = \l_4$ and $n=8$. But for $n \leqs 9$ it is easy to check that $|\mathcal{M}(H)| \leqs |H|$ and the result follows.
\end{proof}

Next we turn to the unitary groups with $H_0 = {\rm U}_n(2)$. Here the following lemma will be useful.

\begin{lem}\label{l:symmetric}
Let $V = L(\l)$ be a nontrivial $2$-restricted irreducible $\mathbb{F}_2H_0$-module, where $H_0 = {\rm U}_{\ell+1}(2)$ and $\ell \geqs 3$. Then one of the following holds:
\begin{itemize}\addtolength{\itemsep}{0.2\baselineskip}
\item[{\rm (i)}] $\l = \l_1+\l_{\ell}$ and $V = V_{{\rm adj}}$;
\item[{\rm (ii)}] $\ell$ is odd, $\l = \l_{(\ell+1)/2}$ and $\dim V = \binom{l+1}{(l+1)/2}$;
\item[{\rm (iii)}] $\dim V \geqs \frac{1}{4}\ell(\ell^2-1)(\ell-2)$.
\end{itemize}  
\end{lem}

\begin{proof}
Write $\l = \sum_ic_i\l_i$ and recall that $c_i = c_{\ell+1-i}$ for all $1 \leqs i \leqs \lfloor \ell/2\rfloor$. If $\ell \leqs 8$ then the result is easily checked by inspecting the relevant tables in \cite{Lub}, so for the remainder we may assume $\ell \geqs 9$.

Let $\{\a_1, \ldots, \a_{\ell}\}$ be a set of simple roots in the corresponding root system of type $A_{\ell}$, labelled in the same way as the fundamental dominant weights $\{\l_1, \ldots, \l_{\ell}\}$. Let $\mathcal{W} = {\rm Sym}_{\ell+1}$ be the Weyl group and let $\mathcal{W}_{\l}$ be the stabiliser of the highest weight $\l$ with respect to the usual action of $\mathcal{W}$ on the set of weights of $V$. Then $\mathcal{W}_{\l}$ is the parabolic subgroup of $\mathcal{W}$ generated by the reflections along the simple roots $\a_i$ such that $c_i = 0$ and we have
\begin{equation}\label{e:weights}
\dim V \geqs |\mathcal{W} \cdot \l| = |\mathcal{W}:\mathcal{W}_{\l}|
\end{equation}
since $V$ has at least $|\mathcal{W}\cdot \l|$ distinct weight spaces.

Suppose $\ell \geqs 9$ and $\l \ne \l_1+\l_{\ell}, \l_{(\ell+1)/2}$, so there exists an integer $2 \leqs k < (\ell+1)/2$ with $c_k  = c_{\ell+1-k} \ne 0$. Then $\mathcal{W}_{\l} \leqs {\rm Sym}_k \times {\rm Sym}_{\ell-2k+1} \times {\rm Sym}_k$, whence the bound in \eqref{e:weights} yields
\[
\dim V \geqs \frac{(\ell+1)!}{k!k!(\ell-2k+1)!} 
\]
and it is straightforward to show that this lower bound is minimal when $k=2$. This gives the lower bound in (iii) and the proof of the lemma is complete.
\end{proof}

\begin{prop}\label{p:unitary}
The conclusion to Theorem \ref{t:main} holds if $H_0 = {\rm U}_{n}(2)$. 
\end{prop}

\begin{proof}
Write $n = \ell+1$ and recall that the highest weight $\l$ is invariant under the graph automorphism of the corresponding Dynkin diagram of type $A_{\ell}$. For $\ell \leqs 7$, it is easy to verify the bound $|\mathcal{M}(H)| \leqs |H|$ using {\sc Magma}, so in view of Theorem \ref{t:wall}, we may assume that $\ell \geqs 8$. Note that $|H|\leqs |{\rm Aut}(H_0)|<2^{(\ell+1)^2}$.

Define $\mathcal{P}$ and $\a(x)$ as in Section \ref{s:intro}. By \cite[Theorem 4.2]{GS} we have $\a(x) \leqs \ell+1$ for all $x \in \mathcal{P}$, so by appealing to \eqref{e:2d} we deduce that $H$ has a regular orbit on $V$ if 
\begin{equation}\label{e:uni}
2^d > 2^{(\ell+1)^2+\left(1-\frac{1}{\ell+1}\right)d}
\end{equation}
It is easy to verify that this inequality holds if $d \geqs \frac{1}{4}\ell(\ell^2-1)(\ell-2)$, so by combining Proposition \ref{p:special} and Lemma \ref{l:symmetric}, it remains to consider the case where $\ell = 2m-1$ is odd, $\l = \l_m$ and $d = \binom{\ell+1}{m}$. Here one checks that \eqref{e:uni} holds if $\ell \geqs 13$ so we may assume $\ell \in \{9,11\}$. If $\ell=9$ then the maximal subgroups of $H$ are recorded in \cite[Tables 8.62 and 8.63]{BHR} (up to conjugacy) and it is straightforward to check that $|\mathcal{M}(H)| \leqs |H|$. Similarly, we can appeal to \cite[Tables 8.78 and 8.79]{BHR} when $\ell=11$. In both cases, we now apply Theorem \ref{t:wall}.
\end{proof}

\begin{prop}\label{p:symplectic}
The conclusion to Theorem \ref{t:main} holds if $H_0 = {\rm Sp}_{n}(2)$.
\end{prop}

\begin{proof}
If $n \leqs 16$ then we can use the {\sc Magma} function \texttt{ClassicalMaximals} to construct a complete set of representatives of the conjugacy classes of maximal subgroups of $H$ and this allows us to easily verify the bound $|\mathcal{M}(H)| \leqs |H|$. For the remainder, we may assume $n \geqs 18$, so $H = H_0$ and we may write $n = 2\ell$, where $\ell$ denotes the rank of $H$. Define $\mathcal{P}$ and $\a(x)$ as above. If $x \in \mathcal{P}$ then \cite[Theorem 4.3]{GS} implies that either $\a(x) \leqs \ell+3$, or $x$ is a transvection and $\a(x) = 2\ell+1$. Since $H$ contains exactly $2^{2\ell}-1$ transvections and $|H| < 2^{\ell(2\ell+1)}$, we deduce that $H$ has a regular orbit on $V$ if
\begin{equation}\label{e:symp1}
2^d > 2^{\ell(2\ell+1)+\left(1-\frac{1}{\ell+3}\right)d} +(2^{2\ell}-1)\cdot 2^{\left(1-\frac{1}{2\ell+1}\right)d}.
\end{equation}
One can check that this bound is satisfied if $\ell \geqs 9$ and $d \geqs 3\ell^3$. Therefore, we may assume $d < 3\ell^3$. 

If $\ell=9$ and $d<3\ell^3$ then by inspecting \cite{Lub}, excluding any cases handled in Proposition \ref{p:special}, we deduce that $\l \in \{\l_1+\l_2, \l_3, \l_{9}\}$. Now assume $\ell \geqs 10$. Here $3\ell^3<16\binom{\ell}{4}$ and we can use \cite[Theorem 1.2]{Martinez} to determine the possibilities for $V$. In this way, the problem is reduced to the cases where $\l \in \{\l_1+\l_2,\l_3\}$, or $\l = \l_{\ell}$ and $9 \leqs \ell \leqs 12$. We consider the three possibilities for $\l$ in turn. Define $\mathcal{P}$ as above.

\vs

\noindent \emph{Case 1. $\l = \l_{\ell}$, $9 \leqs \ell \leqs 12$.}

\vs

First assume $\l = \l_{\ell}$, so $V = L(\l)$ is the spin module for $H$ and $d = \dim V = 2^{\ell}$ (see \cite[p.195]{KL}). For $x \in \mathcal{P}$, we claim that $\dim C_V(x) = 2^{\ell-1}(1+2^{-j})$ if $x$ is an involution of type $a_{2j}$ (see \cite{AS} for the notation) and $\dim C_V(x) \leqs 2^{\ell-1}$ in all other cases. Since $j \leqs \lfloor \ell/2 \rfloor$ and $|x^H|<2^{4j(\ell-j)+1}$ if $x$ is an involution of type $a_{2j}$ (see \cite[Table 3.4.1]{BG}), we deduce that $H$ has a regular orbit on $V$ if
\[
2^d > 2^{\ell(2\ell+1)+d/2} + \sum_{j=1}^{\lfloor \ell/2 \rfloor}2^{4j(\ell-j)+1+d/2+2^{\ell-j-1}}.
\]
It is easy to check that this inequality is satisfied for all $9 \leqs \ell \leqs 12$. Therefore, to complete the argument in this case it just remains to establish the above claim on $\dim C_V(x)$.

To do this, it will be convenient to work with the ambient simple algebraic group $\bar{H} = {\rm Sp}_{2\ell}(k)$, where $k$ is the algebraic closure of $\mathbb{F}_2$. Set $\bar{V} = V \otimes k$ and consider a connected maximal rank subgroup $\bar{A} = {\rm Sp}_{2}(k)^{\ell}$ of $\bar{H}$ (so $\bar{A}$ is the connected component of the stabiliser in $\bar{H}$ of an orthogonal decomposition of the natural $k\bar{H}$-module into $\ell$ two-dimensional nondegenerate subspaces). Then the restriction of $\bar{V}$ to $\bar{A}$ is irreducible and it is given by the tensor product of the natural modules of the ${\rm Sp}_{2}(k)$ factors. Let $x \in H$ be an element of prime order. If $x$ is semisimple then $x$ is $\bar{H}$-conjugate to an element of $\bar{A}$ (this is clear since $\bar{A}$ contains a maximal torus of $\bar{H}$) and thus $\dim C_V(x) = \dim C_{\bar{V}}(x) \leqs d/2$ by \cite[Lemma 3.7]{LSh}. Similarly, if $x$ is an involution of type $b$ or $c$, then some $\bar{H}$-conjugate of $x$ is contained in $\bar{A}$ and the same conclusion holds. 

Finally, let us assume $x$ is an involution of type $a_{2j}$, where $1 \leqs j \leqs \lfloor \ell/2 \rfloor$. 
These elements require special attention because no $\bar{H}$-conjugate of $x$ is contained in $\bar{A}$. Let $\bar{B}$ be a maximal rank subgroup of $\bar{H}$ of the form ${\rm Sp}_{4}(k)^j \times {\rm Sp}_{2\ell-4j}(k)$. By replacing $x$ by a suitable conjugate, we may assume $x = (x_1, \ldots, x_j, y) \in \bar{B}$, where each $x_i \in {\rm Sp}_{4}(k)$ is an $a_2$-type involution and $y \in {\rm Sp}_{2\ell-4j}(k)$ is the identity element. The restriction of $\bar{V}$ to $\bar{B}$ is irreducible, given by the tensor product of $j$ spin modules for ${\rm Sp}_{4}(k)$, together with the spin module for ${\rm Sp}_{2\ell-4j}(k)$ if $j \ne \ell/2$ (recall that if $\{\rho_1,\ldots, \rho_m\}$ is a set of fundamental dominant weights for ${\rm Sp}_{2m}(k)$, labelled in the usual manner, then the spin module has highest weight $\rho_m$). Each $x_i$ has Jordan form $(J_2,J_1^2)$ on the spin module for ${\rm Sp}_{4}(k)$, where $J_i$ denotes a standard unipotent Jordan block of size $i$ (note that the spin module is a twist of the natural $4$-dimensional module for ${\rm Sp}_4(k)$ by a graph automorphism, which interchanges long and short root elements). Therefore, $x$ has Jordan form
\[
(J_2,J_1^2) \otimes \cdots \otimes (J_2,J_1^2) \otimes (J_1^{2^{\ell-2j}}) = (J_2^{2^{\ell-1}(1-2^{-j})}, J_1^{2^{\ell-j}})
\]
on $\bar{V}$ and thus $\dim C_{V}(x) = 2^{\ell-1}(1+2^{-j})$ as claimed.

\vs

\noindent \emph{Case 2. $\l = \l_1+\l_2$, $\ell \geqs 9$.}

\vs

Here $d = \dim V = 16\binom{l+1}{3}$ (see \cite[Table 3]{Martinez}) and one checks that \eqref{e:symp1} holds if $\ell \geqs 12$. Therefore, we may assume $\ell \in \{9,10,11\}$. By combining the proof of \cite[Proposition 2.4.7]{GL} with \cite[Lemma 1.3.2]{GL} and \cite[Proposition 2.2.3]{GL}, we deduce that $\dim C_V(x) \leqs d- 4(\ell-1)^2$ for all $x \in \mathcal{P}$. Therefore, $H$ has a regular orbit on $V$ if 
\[
2^d>2^{\ell(2\ell+1)+d-4(\ell-1)^2}
\]
and it is clear that this is satisfied for each $\ell \in \{9,10,11\}$.

\vs

\noindent \emph{Case 3. $\l = \l_{3}$, $\ell \geqs 9$.}

\vs

Finally, let us assume $\l = \l_3$. Here $d = \binom{2\ell}{3}-2(1+\delta)\ell$, where $\delta = 1$ if $\ell$ is odd, otherwise $\delta = 0$, and the inequality in \eqref{e:symp1} fails to hold for all $\ell \geqs 9$. Let $x\in H$ be of prime order $r$ and set $\bar{H} = {\rm Sp}_{2\ell}(k)$ as above. By examining the proof of \cite[Proposition 2.5.22]{GL} and applying \cite[Lemma 1.3.2]{GL}, we deduce that
\begin{equation}\label{e:short}
\dim C_V(x)\leqs d-4\ell^2+16\ell-14
\end{equation}
if $r \ne 2$, or if $r=2$ and the Zariski closure of $x^{\bar{H}}$ contains a short root element (that is, an element of type $a_2$ in the notation of \cite{AS}). By considering the closure relation on unipotent classes in $\bar{H}$ (see \cite{Spal}, for example), we see that if $x$ is an involution then the closure of $x^{\bar{H}}$ contains a short root element unless $x$ is a long root element (that is, unless $x$ is a transvection). In this exceptional case, the proof of \cite[Proposition 2.5.22]{GL} gives 
\begin{equation}\label{e:long}
\dim C_V(x) \leqs d- 2\ell^2+5\ell-1
\end{equation}
and we noted above that $H$ contains precisely $2^{2\ell}-1$ transvections. Therefore, $H$ has a regular orbit on $V$ if 
\[
2^d > 2^{\ell(2\ell+1)+d-4\ell^2+16\ell-14}+(2^{2\ell}-1) \cdot 2^{d-2\ell^2+5\ell-1}
\]
and it is straightforward to check that this bound holds for all $\ell\geqs 9$.
\end{proof}

\begin{prop}\label{p:orthog}
The conclusion to Theorem \ref{t:main} holds if $H_0 = \O_{n}^{\e}(2)$.
\end{prop}

\begin{proof}
Write $n = 2\ell$, where $\ell$ is the rank of $H_0$. For $\ell \leqs 8$ we can use {\sc Magma} to compute $|\mathcal{M}(H)|$, so we may assume $\ell \geqs 9$. Let $x \in \mathcal{P}$. Then  \cite[Theorem 4.4]{GS} states that either $\a(x) \leqs \ell+3$, or $H = {\rm O}_{n}^{\e}(2)$, $x$ is a transvection and $\a(x) = 2\ell$. There are precisely $2^{\ell-1}(2^{\ell}-\e)$ transvections in ${\rm O}_{n}^{\e}(2)$, whence $H$ has a regular orbit on $V$ if
\begin{equation}\label{e:ort1}
2^d > 2^{\ell(2\ell-1)+1+\left(1-\frac{1}{\ell+3}\right)d} +2^{\ell-1}(2^{\ell}+1)\cdot 2^{\left(1-\frac{1}{2\ell}\right)d}.
\end{equation}
One checks that this bound is satisfied if $\ell \geqs 9$ and $d \geqs \frac{8}{3}\ell^3$. Therefore, we may assume $d< \frac{8}{3}\ell^3$. Now for $\ell \geqs 9$ we have $\frac{8}{3}\ell^3<16\binom{\ell}{4}$ and so we can apply \cite[Theorem 1.2]{Martinez} to determine the possibilities for $V = L(\l)$. In this way, we reduce to the cases $\l \in \{ \l_1+\l_2, \l_3\}$, together with the spin modules $\l \in \{\l_{\ell-1},\l_{\ell}\}$ for $9 \leqs \ell \leqs 13$.

\vs

\noindent \emph{Case 1. $\l \in \{ \l_{\ell-1}, \l_{\ell}\}$, $9 \leqs \ell \leqs 13$.}

\vs

Here $V$ is a spin module and $d = \dim V = 2^{\ell-1}$. Note that $\e=+$ and $H = H_0$ because $\l$ is not invariant under the corresponding graph automorphism of the root system of type $D_{\ell}$. In particular, $H$ does not contain any transvections. The case $\ell=9$ is handled in \cite{BT} (see the proof of \cite[Proposition 3.4]{BT}), so we may assume $10 \leqs \ell \leqs 13$. Let $W$ be the natural module for $H$ and let $x \in H$ be an element of prime order. 

First assume $x$ fixes a nonsingular vector in $W$. In this situation, we may embed $x$ in a maximal subgroup $A = {\rm Sp}_{2\ell-2}(2)$ of $H$ (namely, the stabiliser in $H$ of a nonsingular vector in $W$). Let $\{\rho_1,\ldots, \rho_{\ell-1}\}$ be a set of fundamental dominant weights for $A$, labelled in the usual way. Then the restriction of $V$ to $A$ is irreducible with highest weight $\rho_{\ell-1}$, so we can appeal to the argument in Case 1 in the proof of Proposition \ref{p:symplectic}. In this way, we deduce that either $\dim C_V(x) \leqs 2^{\ell-2}$, or $x = a_{2j}$ is an $a$-type involution with $1 \leqs j \leqs \lceil \ell/2\rceil-1$ and $\dim C_V(x) = 2^{\ell-2}(1+2^{-j})$. In the latter case, let us also note that $|x^H|<2^{2j(2\ell-2j-1)+1}$ (see \cite[Table 3.5.1]{BG}).

Now suppose $x$ does not fix a nonsingular vector in $W$. Then either $x$ is a semisimple element with a trivial $1$-eigenspace on $W$, or $\ell$ is even and $x$ is ${\rm O}_{2\ell}^{+}(2)$-conjugate to an involution of type $a_{\ell}$. To handle these cases, let $\bar{H} = {\rm SO}_{2\ell}(k)$ be the ambient simple algebraic group, where $k$ is the algebraic closure of $\mathbb{F}_2$, and set $\bar{V} = V \otimes k$ and $\bar{W} = W \otimes k$. We first embed $x$ in a connected maximal rank subgroup $\bar{B} = {\rm SO}_8(k) \times {\rm SO}_{2\ell-8}(k)$ of $\bar{H}$, say $x = (x_1,x_2)$ with $x_1 \in {\rm SO}_8(k)$ and $x_2 \in {\rm SO}_{2\ell-8}(k)$. Then 
\[
\bar{V} \downarrow \bar{B} = (U_1 \otimes W_1) \oplus (U_2 \otimes W_2)
\]
describes the restriction of $\bar{V}$ to $\bar{B}$, where $U_1$ and $U_2$ are the two non-isomorphic $8$-dimensional spin modules for ${\rm SO}_8(k)$ and $W_1$, $W_2$ are the non-isomorphic spin modules for ${\rm SO}_{2\ell-8}(k)$ of dimension $2^{\ell-5}$. Set $V_i = U_i \otimes W_i$ and write $s_i$ for the codimension of the largest eigenspace of $x_1$ on $U_i$. We claim that $\dim C_V(x) \leqs 5d/8$.

To see this, first assume $x \in H$ is semisimple with a trivial $1$-eigenspace on $W$. Notice that either $s_i \geqs 4$, or $x_1$ has a $6$-dimensional $1$-eigenspace on $U_i$ and $s_i = 2$. Since $x_1$ has a $6$-dimensional eigenspace on at most one of the two spin modules for ${\rm SO}_8(k)$, we may assume that $s_1 \geqs 4$. Therefore, by applying \cite[Lemma 3.7]{LSh} we deduce that $\dim C_{V_1}(x) \leqs \dim V_1 - 2^{\ell-3}$ and $\dim C_{V_2}(x) \leqs \dim V_2 - 2^{\ell-4}$, whence
\[
\dim C_V(x) = \dim C_{\bar{V}}(x) = \dim C_{V_1}(x) + \dim C_{V_2}(x) \leqs \frac{1}{4}\dim V + \frac{3}{8}\dim V = \frac{5}{8}d
\]
as claimed. A very similar argument applies if $\ell$ is even and $x$ is an involution of type $a_{\ell}$. Here $x_1$ and $x_2$ are involutions of type $a_4$ and $a_{\ell-4}$, respectively, and we note that $x_1$ has Jordan form $(J_2^4)$ on one of the spin modules for ${\rm SO}_8(k)$ and $(J_2^2,J_1^4)$ on the other. Therefore, we may assume $s_1 = 4$, $s_2=2$ and the previous argument applies.

This justifies the claim. Since $|H_0|<2^{\ell(2\ell-1)}$, we conclude that $H$ has a regular orbit on $V$ if
\[
2^d > 2^{\ell(2\ell-1)+5d/8} + \sum_{j=1}^{\lceil \ell/2 \rceil-1}2^{2j(2\ell-2j-1)+1+d/2+2^{\ell-j-2}}.
\]
One checks that this inequality holds for all $\ell \geqs 10$, as required.

\vs

\noindent \emph{Case 2. $\l  = \l_{1}+ \l_{2}$, $\ell \geqs 9$.}

\vs

Here $d = 16\binom{l+1}{3}$ and one checks that the bound in \eqref{e:ort1} holds for all $\ell \geqs 9$. In particular, $H$ has a regular orbit on $V$ and thus $G$ is not extremely primitive.

\vs

\noindent \emph{Case 3. $\l = \l_3$, $\ell \geqs 9$.}

\vs

Finally, suppose $\l = \l_3$ so $d = \binom{2\ell}{3}-2(1+\delta)\ell$, where $\delta = 1$ if $\ell$ is odd, otherwise $\delta = 0$. The inequality in \eqref{e:symp1} fails to hold for all $\ell \geqs 9$. 

Let $x \in H$ be an element of prime order. Define $\bar{H} = {\rm O}_{2\ell}(k)$, where $k$ is the algebraic closure of $\mathbb{F}_2$, and consider the usual embedding of $\bar{H}$ in the symplectic group $\bar{L} = {\rm Sp}_{2\ell}(k)$. Then we may view $\bar{V} = V \otimes k$ as an irreducible module for $\bar{L}$, with irreducible restriction to the connected component $\bar{H}^0 = {\rm SO}_{2\ell}(k)$. This allows us to proceed as in Case 3 in the proof of the previous proposition, using \cite[Proposition 2.5.22]{GL} to obtain bounds on $\dim C_V(x)$. Specifically, if $x \in H \setminus H_0$ is a $b_1$-type involution, then $x$ embeds in $\bar{L}$ as a long root element and thus \eqref{e:long} holds, while the bound in \eqref{e:short} is valid for all other elements $x \in H$ of prime order. 

As noted above, there are at most $2^{\ell-1}(2^{\ell}+1)$ distinct $b_1$-type involutions in $H$. Since $|H|<2^{\ell(2\ell-1)+1}$, we deduce that $H$ has a regular orbit on $V$ if
\[
2^d > 2^{\ell(2\ell-1)+1+d-4\ell^2+16\ell-14}+2^{\ell-1}(2^{\ell}+1) \cdot 2^{d-2\ell^2+5\ell-1}.
\]
It is routine to check that this inequality holds for all $\ell\geqs 9$.
\end{proof}

\vs

This completes the proof of Theorem \ref{t:main}.

\end{document}